\documentclass[a4paper,12pt,openany]{amsart}
\usepackage{indentfirst}
\usepackage{fancyhdr}
\usepackage{amssymb}
\usepackage{amsmath} 
\usepackage{latexsym}
\usepackage{amsthm} 
\usepackage{eucal} 
\usepackage{enumerate} 
\usepackage{pdfpages}
\usepackage{mathrsfs}
\usepackage{stackrel}

\usepackage{graphicx}
\usepackage{fancyhdr}

\theoremstyle{plain}

\newtheorem{theorem}{Theorem}[section]

\newtheorem{prop}[theorem]{Proposition}
\usepackage{enumerate}
\theoremstyle{definition}
\newtheorem{dfn}[theorem]{Definition}
\newtheorem{oss}[theorem]{Remark}
\newcommand{\ve}{\varepsilon}
\newcommand{\R}{\mathbb{R}}

\begin{document}

\title[$p-$Neumann Fractional Laplacians]{Neumann fractional $p-$Laplacian: eigenvalues and existence results}

\author[Dimitri Mugnai]{Dimitri Mugnai}
\author[Edoardo Proietti Lippi] {Edoardo Proietti Lippi}

\address[D. Mugnai]{Department of Ecology and Biology (DEB) \newline\indent
	 Tuscia University\newline \indent
	Largo dell'Universit\`a, 01100 Viterbo, Italy}
\email{dimitri.mugnai@unitus.it}

\address[E. Proietti Lippi]{Department of Mathematics and Computer Science \newline\indent University of Florence\newline\indent
Viale Morgagni 67/A, 50134 Firenze - Italy}
\email{edoardo.proiettilippi@unifi.it}

\maketitle
\begin{abstract}
We develop some properties of the $p-$Neumann derivative for the fractional $p-$Laplacian in bounded domains with general $p>1$. In particular, we prove the existence of a diverging sequence of eigenvalues and we introduce the evolution problem associated to such operators, studying the basic properties of solutions. Finally, we study a nonlinear problem with source in absence of the Ambrosetti-Rabinowitz condition.
\end{abstract}


2010AMS Subject Classification: 35A15, 47J30, 35S15, 47G10, 45G05.

\section{Introduction}

Consider a bounded domain $\Omega$ of $\R^N$, $N\geq1$, with Lipschitz boundary. The aim of this paper is to investigate problems of the form
\begin{equation}\label{pbmod}
\begin{cases}
(-\Delta)^s_p u =f(x,u) \quad $ in $ \Omega,
\\
\mathscr{N}_{s,p}u=g(x)  \quad \quad $ in $ \R^N \setminus \overline{\Omega},
\end{cases}
\end{equation}
where 
\begin{equation}\label{plap}
 (-\Delta)^s_p u(x)=C_{N,s,p} PV\int_{\R^N} |u(x)-u(y)|^{p-2}\frac{u(x)-u(y)}{|x-y|^{N+ps}}\,dy
\end{equation}
is the fractional $p$-Laplacian and
\begin{equation}\label{ns}
\mathscr{N}_{s,p}u(x) := C_{N,s,p} \int_\Omega |u(x)-u(y)|^{p-2}\frac{u(x)-u(y)}{|x-y|^{N+ps}}\,dy ,\quad x\in \R^N \setminus \overline{\Omega},
\end{equation}
is the {\em nonlocal normal $p-$derivative}, or  {\em $p-$Neumann boundary condition} and describes the natural Neumann boundary condition in presence of the fractional $p-$Laplacian. It extends the notion of nonlocal normal derivative introduced in \cite{DPROV} for the fractional Laplacian, i.e. for $p=2$. In our situation, $p>1$, $s\in(0,1)$ and $C_{N,s,p}$ is the constant appearing in the definition of the fractional $p-$Laplacian; however, for the sake of simplicity, from now on, we will set $C_{N,s,p}=1$.

The definition in \eqref{ns} was introduced in \cite{BMPS}, where basic integration by parts were given. Here, we present some further properties of the associated operator, following \cite{DPROV}, where a detailed description of the case $p=2$ was given. Indeed, we refer to \cite{DPROV} for several comments, justifications and reasons to consider such operators, and for this reason we shall skip these motivations; see also \cite{molicalibro} for a general overview on fractional operators.

We shall also face the parabolic problem associated to this new class of operators, namely
\[
\begin{cases}
u_t(x,t)+(-\Delta)^s_pu(x,t)=0 &\quad \text{in } \Omega, \quad t>0 \\
\mathscr{N}_{s,p}u(x,t)=0 &\quad \text{in } \R^N\setminus \overline{\Omega}, \quad t>0 \\
u(x,0)=u_0(x) &\quad \text{in } \Omega. 
\end{cases}
\]
In this case, we will prove conservation of the mass and monotony of the associated energy, as in \cite{DPROV}. Investigations on parabolic equations in presence of the fraction $p-$Laplacian have started in recent years, but only in presence of Dirichlet boundary conditions, and there are not many contributions, yet, see for instance \cite{AABP}, \cite{GT}, \cite{vazquez}, \cite{warmapar}. On the other hand, \cite{DPROV} is the first paper where linear parabolic problems with the associated boundary condition are considered, and, in this direction, we intend to introduce the nonlinear case with the associated nonlinear Neumann conditions. We recall that Neumann boundary problems for the $p-$Laplacian were already introduced in \cite{MRT}, but the underlying operator was different from ours, since in their integral definition of fractional Laplacian only points in $\Omega$ were taken into account; more important, their Neumann boundary condition is a pointwise one, like that of \cite{coelrowo}, \cite{coelrowo2}, \cite{coelrowo3}, \cite{mopeve} and \cite{stvo}.

After these preliminary, but natural, properties, we will consider problem \eqref{pbmod} first with a given source, just to treat the easy case. Then, we will study \eqref{pbmod} in presence of a general nonlinear term which doesn't satisfy the usual Ambrosetti-Rabinowitz condition, showing the existence of two solutions, one being positive in the whole of $\R^N$, and the other being negative.
\medskip

The paper is organized as follows. In Section 2 we consider the variational setting for the nonlocal elliptic problem
associated to the $p-$Neumann boundary condition, recalling some properties from \cite{BMPS} and proving a maximum principle. In addition, we prove that the $p-$Neumann boundary condition is also valid pointwise (see Theorem~\ref{boundary}).

In Section 3 we consider the associated eigenvalue problem. In particular, we prove the existence of an unbounded sequence of eigenvalues and we show that some classical properties of the set of eigenvalues for the $p-$Laplacian still hold true in this case. In particular, we show that any eigenfunction is bounded in the whole of $\R^N$.

In Section 4 we consider the associated parabolic problem and we show that, as in the classical case, the total mass is preserved and the energy is decreasing in time.

Finally, in Section 5, after treating the easy problem with an assigned source, we study a general problem where the right hand side function doesn't satisfy the Ambrosetti-Rabinowitz condition, and we show the the existence of two constant sign solutions by variational methods.

\section{Functional setting for the normal $p-$derivative}

In this section we follow the lines of \cite{DPROV}, introducing the functional setting and the basic properties of the fractional $p-$Laplacian with associated $p-$Neumann boundary conditions.

To do that, fix a bounded domain with Lipschitz boundary $\Omega\subset \R^N$,  $N\geq1$, and for $u:\R^N \to \R$ measurable, set
\[
\|u\|_X:=\left(\|u\|_{L^p(\Omega)}^p+  \| |g|^\frac{1}{p}u \|^p_{L^p(\R^N \setminus \Omega)}+
 \int_{\R^{2N}\setminus(C\Omega)^2}\frac{|u(x)-u(y)|^p}{|x-y|^{N+ps}}\,dxdy \right)^\frac{1}{p},
\]
where $C\Omega=\R^N\setminus \Omega$, and
$$ X:=\lbrace u:\R^N \to \R\quad \text{measurable such that }\|u\|_X<\infty \rbrace .$$

\begin{oss}
It is clear that, $\Omega$ being ``nice enough'', in the previous setting we can equally write $\R^N\setminus \Omega$ in place of $\R^N\setminus \overline\Omega$. The abstract setting can be faced also for $\Omega$ less regular, replacing $\||g|^{\frac{1}{p}}u\|_{L^p(\R^N\setminus \Omega)}$ with $\||g|^{\frac{1}{p}}u\|_{L^p(\R^N\setminus \overline \Omega)}$, which is the natural norm in the general framework.\end{oss}

Though already stated in \cite{BMPS}, we recall the following result, giving a detailed proof.
\begin{prop}
$X$ is a reflexive Banach space with norm $\|\cdot\|_X $.
\end{prop}
\begin{proof}
First, we show that $\|\cdot\|_X$ is a norm. If $\|u\|_X=0$, we have $\|u\|_{L^p(\Omega)}=0$, so $u=0$ a.e. in $\Omega$.
Moreover, we have
$$ \int_{\R^{2N}\setminus(C\Omega)^2}\frac{|u(x)-u(y)|^p}{|x-y|^{N+ps}}\,dxdy=0,$$
hence $|u(x)-u(y)|=0$ in $\R^{2N}\setminus(C\Omega)^2$. In particular, we can take $x\in C\Omega$ and $y\in \Omega$ to obtain
$$u(x)=u(x)-u(y)=0.$$
In this way, we have $u=0$ a.e. in $\R^N$.

Now, we prove that $X$ is complete, and to do this we take a Cauchy sequence $(u_k)_k$ in $X$.
In particular, $u_k$ is a Cauchy sequence in $L^p(\Omega)$ and so (up to a subsequence) there exists $u\in L^p(\Omega)$ such that $u_k$ converges to
$u$ in $L^p(\Omega)$ and a.e. in $\Omega$. This means that there exists $Z_1\subset \Omega$ such that
\begin{equation}\label{z1}
|Z_1|=0 \text{ and}\, u_k(x)\to u(x) \text{ for every}\, x\in \Omega \setminus Z_1.
\end{equation}
We also define for every $U:\R^N\to \R$ and $(x,y)\in \R^{2N}$
\[
T_U(x,y):=\frac{(U(x)-U(y))\chi_{\R^{2N}\setminus(C\Omega)^2(x,y)}}{|x-y|^{N/p+s}},
\]
so
$$ T_{u_k}(x,y)-T_{u_h}(x,y)= 
\frac{(u_k(x)-u_h(x)-u_k(y)+u_h(y))\chi_{\R^{2N}\setminus(C\Omega)^2(x,y)}}{|x-y|^{N/p+s}}. $$
Since $u_k$ is a Cauchy sequence in $X$, for every $\varepsilon>0$ there exists $N_\varepsilon>0$ such that for 
$h,k\geq N_\varepsilon$ we have in particular
$$\varepsilon^p\geq \int_{\R^{2N}\setminus(C\Omega)^2}\frac{|u_k(x)-u_h(x)-u_k(y)+u_h(y)|^p}{|x-y|^{N+ps}}\,dxdy
=\|T_{u_k}-T_{u_h}\|_{L^p(\R^{2N})}^p.  $$
So, $T_{u_k}$ is a Cauchy sequence in $L^p(\R^{2N})$, and up to a subsequence we can assume that $T_{u_k}$
converges to some $T$ in $L^p(\R^{2N})$ and a.e. in $\R^{2N}$. This means that there exists $Z_2\subset \R^{2N}$ such that
\begin{equation}\label{z2}
|Z_2|=0 \text{ and}\, T_{u_k}(x,y)\to T_u(x,y) \text{ for every}\, (x,y)\in \R^{2N} \setminus Z_2.
\end{equation}
For any $x\in \Omega$, we set
$$S_x:=\lbrace y\in \R^N :\, (x,y)\in \R^{2N}\setminus Z_2  \rbrace, $$
$$W:=\lbrace (x,y)\in \R^{2N}:\, x\in \Omega \text{ and}\,\, y\in \R^N\setminus S_x  \rbrace, $$
$$V:=\lbrace x\in \Omega : \, |\R^N\setminus S_x|=0  \rbrace. $$
If we take $(x,y)\in W$, we have $y\in \R^N\setminus S_x$, so $(x,y) \notin \R^{2N}\setminus Z_2$ that is 
$(x,y)\in Z_2$. From this we get 
\begin{equation}\label{w}
W \subseteq Z_2.
\end{equation}
From \eqref{w} and \eqref{z2}, we obtain $|W|=0$, so by the Fubini's Theorem we have
$$0=|W|=\int_\Omega |\R^{N}\setminus S_x|\,dx, $$
which implies that $|\R^{N}\setminus S_x|=0$ a.e. $x\in \Omega$. It follows that $|\Omega\setminus V|=0$. This together 
with \eqref{z1} implies that
$$|\Omega\setminus (V\setminus Z_1)|=|(\Omega\setminus V)\cup Z_1|\leq|\Omega\setminus V|+|Z_1|=0.$$
In particular, $V\setminus Z_1 \neq \emptyset$ (nay, $|V\setminus Z_1|=|\Omega|$), so we can take $x_0\in V\setminus Z_1$. From \eqref{z1} we have
$$\lim_{k\to \infty} u_k(x_0)=u(x_0). $$
In addition, since $x_0\in V$, we get $|\R^{N}\setminus S_{x_0}|=0$. This means that for a.e. $y\in \R^N$, 
$(x_0,y)\in \R^{2N}\setminus Z_2$ and so 
$$\lim_{k\to \infty}T_{u_k}(x_0,y)=T(x_0,y). $$
Moreover, since $\Omega\times(C\Omega)\subseteq \R^{2N}\setminus(C\Omega)^2$, we have
$$T_{u_k}(x_0,y):= \frac{u_k(x_0)-u_k(y)}{|x_0-y|^{N/p+s}}$$
for a.e. $y\in C\Omega$. From this, we obtain
\begin{align*}
\lim_{k\to \infty}u_k(y)&=\lim_{k\to \infty}\left( u_k(x_0)-|x_0-y|^{N/p+s} T_{u_k}(x_0,y) \right) \\
&=u(x_0)-|x_0-y|^{N/p+s} T(x_0,y)
\end{align*}
for a.e. $y\in C\Omega$. This and \eqref{z1} imply that $u_k$ converges a.e. in $\R^N$, so we can say that $u_k$ converges a.e.
to some $u$ in $\R^N$. Now, since $u_k$ is a Cauchy sequence in $X$, for any $\varepsilon>0$ there exists $N_\varepsilon>0$
such that, for any $h\geq N_\varepsilon$,
\begin{align*}
\varepsilon^p&\geq \liminf_{k\to \infty}\|u_h-u_k\|_X^p  \\
&\geq \liminf_{k\to \infty}\int_{\Omega}|u_h-u_k|^p\,dx+ \liminf_{k\to \infty}\int_{\R^N\setminus \Omega}|g||u_h-u_k|^p\,dx\\
&\quad+ \liminf_{k\to \infty}\int_{\R^{2N}\setminus(C\Omega)^2}\frac{|(u_k-u_h)(x)-(u_k+u_h)(y)|^p}{|x-y|^{N+ps}}\,dxdy \\
&\geq \int_{\Omega}|u_h-u|^p\,dx+ \int_{\R^N\setminus \Omega}|g||u_h-u|^p\,dx\\
&+
\int_{\R^{2N}\setminus(C\Omega)^2}\frac{|(u_k-u)(x)-(u_k+u)(y)|^p}{|x-y|^{N+ps}}\,dxdy \\
&=\|u_h-u\|_X^p,
\end{align*}  
where we used Fatou's Lemma. So $u_h$ converges to $u$ in $X$. Starting this procedure with a generic subsequence, we can conclude that $X$ is complete.

As for the reflexivity, see \cite{BMPS}.
\end{proof}

\begin{oss}
From the definition of $X$, it follows that $X$ is embedded in $L^p(B(0,R))$ for every $R>0$. Indeed, by the convergence of the double integral, we get that for a.e. $x\in \Omega$
\[
\int_{\R^{N}}\frac{|u(x)-u(y)|^p}{|x-y|^{N+ps}}\,dy <\infty,
\]
and so for every $R>0$
\[
\frac{1}{R^{N+ps}}\int_{B(x,R)}|u(x)-u(y)|^p\,dy <\infty.
\]
In addition, we have
\[
\int_{B(x,R)}|u(y)|^p\,dy \leq 2^{p-1}\int_{B(x,R)}|u(x)-u(y)|^p\,dy+2^{p-1}|u(x)|^p|B(x,R)|<\infty,
\]
hence the claim follows.
\end{oss}

\begin{oss}\label{immer}
Under the previous setting, $X$ is embedded continuously in $W^{s,p}(\Omega)$. As a consequence, the standard compact embeddings in suitable $L^q(\Omega)$ spaces hold true, see \cite{DNPV}.
\end{oss}

Now, we recall the analogous of the divergence theorem and of  the integration by parts formula for the nonlocal case, see \cite{BMPS}:
\begin{prop}\label{parti1}
Let $u$ be any bounded $C^2$ function in $\R^N$. Then,
$$\int_\Omega (-\Delta)^s_p u\,dx = -\int_{\R^N \setminus \Omega}\mathscr{N}_{s,p}u\,dx. $$
\end{prop}

\begin{prop}\label{parti}
Let $u$ and $v$ be bounded $C^2$ functions in $\R^N$. Then,
\begin{align*}
\frac{1}{2}\int_{\R^{2N}\setminus(C\Omega)^2}&|u(x)-u(y)|^{p-2}\frac{(u(x)-u(y))(v(x)-v(y))}{|x-y|^{N+ps}}\,dxdy \\
&=\int_\Omega v (-\Delta)^s_p u\,dx+ \int_{\R^N \setminus \Omega}v\mathscr{N}_{s,p}u\,dx. 
\end{align*}
\end{prop}

The integration by parts formula in Proposition \ref{parti} leads to this natural definition:
\begin{dfn}
Let $f\in L^{p'}(\Omega)$ and $g\in L^1(\R^N \setminus \overline \Omega)$. We say that $u\in X$ is a weak solution of 
\begin{equation}\label{probg}
\begin{cases}
(-\Delta)^s_p u =f \quad $ in $ \Omega,
\\
\mathscr{N}_{s,p}u=g  \quad \quad $ in $ \R^N \setminus \overline{\Omega},
\end{cases}
\end{equation}
whenever
\begin{equation}\label{weakg}
\frac{1}{2}\int_{\R^{2N}\setminus(C\Omega)^2}\frac{J_p(u(x)-u(y))(v(x)-v(y))}{|x-y|^{N+ps}}\,dxdy =
\int_\Omega fv\,dx + \int_{\R^N\setminus \overline \Omega} gv\,dx
\end{equation}
for every $v\in X$, where
\[
J_p(u(x)-u(y)):=|u(x)-u(y)|^{p-2}(u(x)-u(y)). 
\]
\end{dfn}

As a consequence of this definition, we have the following result
\begin{theorem}\label{boundary}
Let $u$ be a weak solution of \eqref{probg}. Then, $\mathscr{N}_{s,p}u=g$ a.e. in $\R^N\setminus \overline \Omega$.
\end{theorem} 

\begin{proof}
First, we take $v\in X$ such that $v\equiv 0$ in $\Omega$ as a test function in \eqref{weakg}, obtaining
\begin{align*}
\int_{\R^N\setminus \overline \Omega} gv\,dx&=
-\frac{1}{2}\int_\Omega \int_{\R^N\setminus \overline \Omega}\frac{J_p(u(x)-u(y))v(y)}{|x-y|^{N+ps}}\,dydx \\
&\quad +\frac{1}{2}\int_{\R^N\setminus \overline \Omega}\int_\Omega \frac{J_p(u(x)-u(y))v(x)}{|x-y|^{N+ps}}\,dydx \\
&=-\int_\Omega \int_{\R^N\setminus \overline \Omega}\frac{J_p(u(x)-u(y))v(y)}{|x-y|^{N+ps}}\,dydx \\
&=-\int_{\R^N\setminus \overline \Omega}v(y)\int_\Omega \frac{J_p(u(x)-u(y))}{|x-y|^{N+ps}}\,dxdy \\
&=-\int_{\R^N\setminus \overline \Omega}v(y)\mathscr{N}_{s,p}u(y)\,dy.
\end{align*}
Therefore,
\[
\int_{\R^N\setminus \overline \Omega} (\mathscr{N}_{s,p}u(x)-g(x))v(x)\,dx=0
\]
or every $v\in X$ which is 0 in $\Omega$. In particular, this is true for every $v\in C^\infty_c(\R^N\setminus \overline \Omega)$, and so $\mathscr{N}_{s,p}u(x)=g(x)$ a.e. in $\R^N\setminus \overline \Omega$.
\end{proof}

From the definition of weak solution, we have the following
\begin{prop}
Let $f\in L^{p'}(\Omega)$ and $g\in L^1(\R^N \setminus \Omega)$. Let $I_g:X\to \R$ be the functional defined as
$$I_g(u):=\frac{1}{2p}\int_{\R^{2N}\setminus(C\Omega)^2}\frac{|u(x)-u(y)|^p}{|x-y|^{N+ps}}\,dxdy
-\int_\Omega fu\,dx - \int_{\R^N\setminus \Omega} gu\,dx $$
for every $u\in X$. Then any critical point of $I_g$ is a weak solution of problem \eqref{probg}. 
\end{prop}

\begin{proof}
We only show that $I_g$ is well defined on $X$. Indeed, if $u\in X$ we have
$$\left| \int_\Omega fu\,dx\right|\leq \|f\|_{L^{p'}(\Omega)}\|u\|_{L^{p}(\Omega)}\leq C\|u\|_X. $$
In addition,
$$\left|\int_{\R^N\setminus \Omega} gu \,dx\right|\leq \int_{\R^N\setminus \Omega} |g|^\frac{1}{p'}|g|^\frac{1}{p} |u|
\leq \|g\|^\frac{1}{p'}_{L^1(\R^N \setminus \Omega)} \||g|^\frac{1}{p}u\|_{L^p(\R^N\setminus \Omega)}
\leq C\|u\|_X.$$
Then, if $u\in X$, we have 
$$|I_g(u)|\leq C\|u\|_X <\infty. $$
The computation of the first variation of $I_g$ is standard.
\end{proof}

The next result gives a sort of maximum principle.

\begin{prop}
Let $f\in L^{p'}(\Omega)$ and $g\in L^1(\R^N \setminus \Omega)$. Let $u\in X$ be a weak solution of \eqref{probg} with
$f\geq 0$ and $g\geq 0$. Then, $u$ is constant.
\end{prop}

\begin{proof}
First, we notice that $v\equiv 1$ belongs to $X$. So, using it as a test function in \eqref{weakg} we obtain
$$0\leq \int_\Omega f\,dx = -\int_{\R^N\setminus \Omega} g\,dx \leq0.$$
Hence, $f=0$ a.e. in $\Omega$ and $g=0$ a.e. in $ \R^N\setminus \Omega$. Now, taking $v=u$ as a test function again in
\eqref{weakg}, we get
$$\int_{\R^{2N}\setminus(C\Omega)^2}\frac{|u(x)-u(y)|^p}{|x-y|^{N+ps}}\,dxdy=0, $$
so $u$ must be constant. 
\end{proof}

From now on, we concentrate on homogeneous boundary conditions, so that $g\equiv0$.

Denoting by $X'$ the dual of $X$, we can define the operator $A:X\to X'$ such that
$$
\begin{aligned}
\langle A(u),v\rangle&= \int_\Omega |u|^{p-2}uv\,dx \\
&+ \int_{\R^{2N}\setminus(C\Omega)^2}\frac{J_p(u(x)-u(y))(v(x)-v(y))}{|x-y|^{N+ps}}\,dxdy
\end{aligned} $$
for all $u,v\in X$. In this way $A$ is ($p-1$)-homogeneous and odd, and such that 
$$\langle A(u),u\rangle= \|u\|_X^p, \quad \quad |\langle A(u),v\rangle| \leq \|u\|_X^{p-1} \|v\|_X. $$
By the uniform convexity of $X$, $A$ satisfies the ($S$) property, that is, for all $(u_n)_n$ in $X$ such that
$u_n \rightharpoonup u$ in $X$ and $\langle A(u_n),u_n-u\rangle \to 0$, then $u_n \to u$ in $X$,
 see \cite[Proposition 1.3]{peago}.

\section{The eigenvalue problem}

In this section we consider the nonlinear eigenvalue problem
\begin{equation}\label{probla}
\begin{cases}
(-\Delta)^s_p u = \lambda |u|^{p-2}u \quad $ in $ \Omega,
\\
\mathscr{N}_{s,p}u=0  \quad \quad $ in $ \R^N \setminus \overline{\Omega},
\end{cases}
\end{equation}
depending on parameter $\lambda\in \R$. If \eqref{probla} admits a weak solution $u\in X$ (notice that now $g\equiv0$), that is 
$$\frac{1}{2} \int_{\R^{2N}\setminus(C\Omega)^2}\frac{J_p(u(x)-u(y))(v(x)-v(y))}{|x-y|^{N+ps}}\,dxdy  
 =\lambda \int_\Omega |u|^{p-2}uv\,dx$$
for all $v\in X$, then we say that {\it $\lambda$ is an eigenvalue of $(-\Delta)^s_p$ with $p-$Neumann boundary conditions and associated $\lambda$-eigenfunction $u$}. As in the classical case, 
we call the set of all the eigenvalues the point spectrum of $(-\Delta)^s_p $ in $X$ and we denote it by $\sigma(s,p)$.

First of all we observe that for $\lambda=0$ constant functions are all $0$-eigenfunctions. Since all the eigenvalues are
obviously non negative, we have that $\lambda_1=0$ is the first eigenvalue. Moreover,
$$ \int_{\R^{2N}\setminus(C\Omega)^2}|u(x)-u(y)|^{p-2}\frac{(u(x)-u(y))(v(x)-v(y))}{|x-y|^{N+ps}}\,dxdy =0 $$
for all $v\in X$ implies $u$ constant, so all the $\lambda_1$-eigenfunctions are just constant functions.

As usual, we can construct a sequence $(\lambda_k)_k$ of eigenvalues for problem \eqref{probla},
analogously to the Dirichlet case treated in \cite{iasq}, setting
\[
\lambda_k= \inf_{A\in \mathcal{F}_k} \sup_{u\in A}\,\frac{[u]^p_{s,p}}{2},
\]
with 
$$ [u]^p_{s,p}=\int_{\R^{2N}\setminus(C\Omega)^2}\frac{|u(x)-u(y)|^p}{|x-y|^{N+ps}}\,dxdy. $$
Here, if $\mathcal{F}$ is the family of all nonempty, closed, symmetric 
subsets of $S=\lbrace u\in X :\,\, \int_\Omega |u|^p=1 \rbrace$, for all $k\in \mathbb{N}$ we have set
$$\mathcal{F}_k= \lbrace A\in \mathcal{F}:\, i(A)\geq k \rbrace, $$
while $i(A)$ is the cohomological index of Fadell and Rabinowitz \cite{fara}.

In order to prove that $\lambda_k$ is an eigenvalue for every $k\in \mathbb{N}$, we proceed in the standard way:
set $\varphi(u)=\frac{[u]^p_{s,p}}{2}$, $I(u)=\|u\|_{L^p(\Omega)}^p$ and let
 $\bar{\varphi}$ be the restriction of $\varphi$ to $S$.

\begin{prop}\label{ps}
The functional $\bar{\varphi}$ satisfies the Palais-Smale condition at any level $c\in \R$.
\end{prop}
\begin{proof}
Let $(u_n)_n\subset S$ and $(\mu_n)_n\subset \R$ be such that $\varphi(u_n)\to c$ as $n\to \infty$ and $\varphi'(u_n)-\mu_n I'(u_n)\to 0$ in $X'$.
We have 
$$ \|u_n\|_X^p= 1+\varphi(u_n) \to 1+c, $$
so $(u_n)_n$ is bounded in $X$. Up to a subsequence, we have $u_n\rightharpoonup u$ in $X$ and $u_n\to u$ in $L^p(\Omega)$ for
some $u\in X$ as $n\to \infty$, see Remark \ref{immer}. In particular, $u\in S$. We also get that $\varphi(u_n)-\mu_n \to 0$, and so $\mu_n \to c$. 
Now, we have
\begin{align*}
|p\langle A(u_n),u_n-u\rangle|&=|\langle I'(u),u_n-u\rangle + \langle \varphi'(u),u_n-u\rangle | \\
&=|\langle I'(u),u_n-u\rangle + \mu_n \langle I'(u),u_n-u\rangle +o(1)|\\
&\leq |1+\mu_n|\|u_n-u\|_{L^p(\Omega)}^p+o(1) \to 0.
\end{align*}
So, by the ($S$) property of $A$, we get that $u_n\to u$ in $X$.
\end{proof}

Now we can give the desired result for the sequence $(\lambda_k)_k$.

\begin{prop}
For all $k\in \mathbb{N}$, $\lambda_k$ is an eigenvalue of \eqref{probla}. In addition, $\lambda_k \to \infty$.
\end{prop}
The proof is standard, see for example the proof of \cite[Proposition 2.2]{iasq}. We also recall that in \cite{delp} a characterization of the second eigenvalue is given, together with the asymptotic for $p\to \infty$.

Now we show that every eigenfunction, except the ones corresponding to the first eigenvalue, changes sign.

\begin{prop}
Let $v\in X$ be a solution to \eqref{probla} such that $v>0$ in $\Omega$. Then $\lambda=0$, hence $v$ is constant.
\end{prop}

\begin{proof}
We assume that $v\in X$ is strictly positive solution of \eqref{probla} such that $I(u)=1$, and take
$u\in X$ a $0$-eigenfunction with $I(u)=1$. We set  $v_\ve(x)=v(x)+\ve$, $u_\ve(x)=u(x)+\ve$ and
$$\sigma_t^\varepsilon(x)=\left(tu_\varepsilon(x)^p+(1-t)v_\varepsilon(x)^p\right)^\frac{1}{p} $$
for $x\in \R^N$, $t\in [0,1]$. It follows that $\sigma_t^\varepsilon \in X$ and
$$\varphi(\sigma_t^\varepsilon)\leq  t\varphi(u)+(1-t)\varphi(v) $$
for all $t\in [0,1]$, see \cite[Lemma 4.1]{frapa}. From this, we have
\begin{equation}\label{prima}
\varphi(\sigma_t^\varepsilon)-\varphi(v)\leq t(\varphi(u)-\varphi(v))=-t\lambda
\end{equation}
for all $t\in [0,1]$ and $\varepsilon$ small enough. Moreover, from the convexity of $\varphi$ we get
\begin{equation}\label{seconda}
\begin{aligned}
&\varphi(\sigma_t^\varepsilon)-\varphi(v)\geq \\
& \frac{p}{2}\int_{\R^{2N}\setminus(C\Omega)^2} J_p(v(x)-v(y))\frac{\sigma_t^\varepsilon(x)-\sigma_t^\varepsilon(y)-(v(x)-v(y))}{|x-y|^{N+ps}}\,dxdy,
\end{aligned}
\end{equation}
for all $t\in [0,1]$ and $\varepsilon$ small enough. Taking $\sigma_t^\varepsilon-v_\varepsilon$ as a test function in the weak formulation of \eqref{probla} for the couple $(v,\lambda)$, we obtain
\begin{equation}\label{terza}
\begin{aligned}
\frac{1}{2}\int_{\R^{2N}\setminus(C\Omega)^2} J_p(v(x)-v(y))
&\frac{\sigma_t^\varepsilon(x)-\sigma_t^\varepsilon(y)-(v_\varepsilon(x)-v_\varepsilon(y))}{|x-y|^{N+ps}}\,dxdy,\\
&=\lambda \int_{\Omega}v(x)^{p-1}(\sigma_t^\varepsilon(x)-v_\varepsilon(x))\,dx.
\end{aligned}
\end{equation}
Finally, from \eqref{prima}--\eqref{terza} we get
\begin{equation}\label{rapporto}
p\lambda \int_{\Omega}v(x)^{p-1}\frac{\sigma_t^\varepsilon(x)-v_\varepsilon(x)}{t}\,dx \leq -\lambda,
\end{equation}
for all $t\in (0,1]$ and $\varepsilon$ small enough. From the concavity of the $p$-th root follows that 
$$ \sigma_t^\varepsilon(x)-v_\varepsilon(x) \geq t(u_\varepsilon(x)-v_\varepsilon(x)) =t(u-v)(x)$$
in $\Omega$. So, we can apply Fatou's Lemma in \eqref{rapporto}, obtaining
\[
\lambda \int_{\Omega} \left(\frac{v(x)}{v_\varepsilon(x)}  \right)^{p-1}(u_\varepsilon(x)^p -v_\varepsilon(x)^p)\, dx \leq -\lambda
\]
for $\varepsilon$ small enough. Since $v>0$ in $\Omega$, from the dominated convergence Theorem and $I(u)=I(v)=1$, when
$\varepsilon\to0^+$ we get
$$0\leq -\lambda.$$
Since all the eigenvalues are non negative, we have $\lambda=0$ and so $v$ belongs to the first eigenspace, as claimed.
\end{proof}

Now we want to prove the boundedness of eigenfunctions in the whole of $\R^N$, starting as in \cite{frapa} to get the bound in $\Omega$, and exploiting the $p-$Neumann condition to get the bound in the complementary set of $\Omega$. More precisely, we have that the $L^\infty-$norm in $\Omega$ estimates the $L^\infty-$norm in the $\R^N\setminus \Omega$.

\begin{prop}
Let $s\in (0,1)$, $p>1$, and $u\in X$ be a solution of \eqref{probla} for some $\lambda\geq0$.
Then $u\in L^\infty (\R^N)$ and
\[
\|u\|_{L^\infty(\R^N)}=\|u\|_{L^\infty(\Omega)}.
\]
\end{prop}
\begin{proof}
First, we prove that $u$ is bounded in $\Omega$, concentrating on the case $ps\leq N$, the case $ps>N$ being trivial by the fractional Morrey-Sobolev embedding. As in \cite{frapa}, we only have to prove that $u_+$ is bounded in $\Omega$,
since both $u_\pm$ are solutions, so we can get a bound for the
negative part in the same way. To do that, it is enough to prove that
\begin{equation}\label{bound}
 \|u\|_{L^\infty(\Omega)}\leq 1 \quad \text{when}\quad \|u\|_{L^p(\Omega)}\leq \delta,
\end{equation}
 where $\delta>0$ is still to be determined. Indeed, we can scale the function verifying 
\eqref{bound}, so there is no restriction in this.

Now, for all $k\geq 0$, we define the function
$$w_k:=(u-(1-2^{-k}))_+, $$see \cite{frapa}, also for the following facts: $w_k\in X$ and
\begin{equation}
\begin{aligned}\label{wk}
w_{k+1}(x) \leq w_k(x) \quad \text{a.e. in} \quad \Omega,\\ 
u(x)<(2^{k+1}-1)w_k(x) \quad \text{for} \quad x\in \lbrace w_{k+1}>0 \rbrace,
\end{aligned}
\end{equation}
and the inclusions 
$$\lbrace w_{k+1}>0 \rbrace\subseteq \lbrace w_k>2^{-(k+1)} \rbrace $$
hold true for every $k\geq 0$. Moreover, for every function $v$
\begin{equation}\label{dis+}
|v(x)-v(y)|^{p-2}(v_+(x)-v_+(y))(v(x)-v(y))\geq |v_+(x)-v_+(y)|^p,
\end{equation}
for all $x,y\in \R^N$.

Now, we want to prove \eqref{bound} using a standard argument relying on estimating the decay of $U_k:=\|w_k\|_{L^p(\Omega)}^p$.
First of all, using \eqref{dis+} with $v=u-(1-2^{-k-1})$ we obtain
\begin{align*}
\|w_{k+1}\|_X^p 
\leq \int_{\R^{2N}\setminus(C\Omega)^2}\frac{J_p(u(x)-u(y))(w_{k+1}(x)-w_{k+1}(y))}{|x-y|^{N+ps}}\,dxdy +U_{k+1}.
\end{align*}
Taking $w_{k+1}$ as a test function in \eqref{probla} and then using \eqref{wk}, we get
\[
\begin{aligned}
\|w_{k+1}\|_X^p &\leq \lambda \int_{\lbrace w_{k+1}>0 \rbrace}|u(x)|^{p-2}u(x)w_{k+1}(x)\,dx +U_{k+1}\notag \\
&\leq (\lambda (2^{k+1}-1)^{p-1}+1)U_k.
\end{aligned}
\]
Using the fractional Sobolev embeddings, as in \cite{frapa}, we get
\[
U_{k+1}\leq c \|w_{k+1}\|^p_X|\lbrace w_{k+1}>0 \rbrace|^\frac{N}{ps},
\]
where $c>0$ depends on $N,p,s$. Proceeding as in \cite{frapa}, we get that $u$ is 
bounded in $\Omega$.

Now, take $x\in \R^N\setminus \overline{\Omega}$. Since $u$ is bounded in $\Omega$, from \eqref{probla} we get
$$u(x)\int_{\Omega}\frac{|u(x)-u(y)|^{p-2}}{|x-y|^{N+ps}}\,dy=\int_{\Omega}\frac{|u(x)-u(y)|^{p-2}u(y)}{|x-y|^{N+ps}}\,dy .$$       
If $u$ is constant, the result is trivial. On the other hand, if $u$ is not constant,  from Theorem \ref{boundary} we have
$$ |u(x)|=
\left|\frac{\displaystyle\int_{\Omega}\frac{|u(x)-u(y)|^{p-2}u(y)}{|x-y|^{N+ps}}\,dy}{\displaystyle\int_{\Omega}\frac{|u(x)-u(y)|^{p-2}}{|x-y|^{N+ps}}\,dy}  \right| \leq \|u\|_{L^\infty(\Omega)},$$
and so $\|u\|_{L^\infty(\R^N\setminus\Omega)}\leq \|u\|_{L^\infty(\Omega)}$, which concludes the proof.
\end{proof}
	
\section{The parabolic equation}

In this section, we consider the problem
\begin{equation}\label{probheat}
\begin{cases}
u_t(x,t)+(-\Delta)^s_pu(x,t)=0 &\quad \text{in } \Omega, \quad t>0 \\
\mathscr{N}_{s,p}u(x,t)=0 &\quad \text{in } \R^N\setminus \overline{\Omega}, \quad t>0 \\
u(x,0)=u_0(x) &\quad \text{in } \Omega. 
\end{cases}
\end{equation}

We show that the solutions of \eqref{probheat} preserve their mass and have energy that decreases in time, as proved in \cite{DPROV} for $p=2$. To do so, we assume that $u$ is a classical solution of \eqref{probheat}, so that \eqref{probheat} holds pointwise. In particular, we can differentiate with respect to time.

\begin{prop}\label{mass}
Let $u$ be a classical solution of \eqref{probheat} such that $u$ is bounded and $|u_t(x,t)|+|(-\Delta)^s_pu(x,t)|\leq K$ 
for all $t>0$ and all $x\in \Omega$. Then, for all $t>0$
$$\int_\Omega u(x,t)\,dx= \int_\Omega u_0(x)\,dx,$$
which means that the total mass is preserved.
\end{prop}
\begin{proof}
By the dominated convergence theorem and Proposition \ref{parti1}, we have
$$\frac{d}{dt}\int_\Omega u\,dx= \int_\Omega u_t\,dx = -\int_\Omega(-\Delta)^s_pu\,dx
=\int_{\R^N\setminus \Omega} \mathscr{N}_{s,p}u\,dx =0.$$
So, $\int_\Omega u\,dx$ does not depend on $t$, as desired.
\end{proof}

\begin{prop}
Under the assumptions of Proposition $\ref{mass}$, the energy
$$E(t)=\int_{\R^{2N}\setminus (C\Omega)^2} \frac{|u(x,t)-u(y,t)|^p}{|x-y|^{N+ps}}\,dxdy$$
is decreasing in time $t>0$.
\end{prop}

\begin{proof}
From Proposition \ref{parti}, we have
\begin{align*}
E'(t)&=\frac{d}{dt} \int_{\R^{2N}\setminus (C\Omega)^2} \frac{|u(x,t)-u(y,t)|^p}{|x-y|^{N+ps}}\,dxdy \\
&=p\int_{\R^{2N}\setminus (C\Omega)^2}J_p(u(x,t)-u(y,t))\frac{u_t(x,t)-u_t(y,t)}{|x-y|^{N+ps}}\,dxdy \\
&=2p \int_\Omega u_t(-\Delta)^s_pu\,dx=-2p \int_\Omega \left|(-\Delta)^s_pu \right|^2\,dx \leq 0,
\end{align*}
since $u$ is a solution of \eqref{probheat}, and so the energy is decreasing.  
\end{proof}

\section{Two $p-$Neumann problems with source}

In this section we consider two problems in presence of the $p-$Neumann condition: the first one is the easy case of a given source term, which we study for completeness of the subject, while the second one takes into account a source not satisfying the Ambrosetti-Rabinowitz condition or some of its standard generalizations (see \cite{addendum}).
\smallskip

Let us start with
\begin{equation}\label{probco}
\begin{cases}
(-\Delta)^s_p u + |u|^{p-2}u=f(x) &\quad $ in $ \Omega,
\\
\mathscr{N}_{s,p}u=0  &\quad $ in $ \R^N \setminus \overline{\Omega},
\end{cases}
\end{equation}
with $f\in L^{p'}(\Omega)$.
\begin{dfn}
We say that $u\in X$ is a weak solution of problem \eqref{probco} if 
$$
\frac{1}{2}\int_{\R^{2N}\setminus(C\Omega)^2}\frac{J_p(u(x)-u(y))(v(x)-v(y))}{|x-y|^{N+ps}}\,dxdy + \int_\Omega |u|^{p-2}uv\,dx
=\int_\Omega fv\,dx$$
for every  function $v\in X$.
\end{dfn}

For the sake of simplicity, in this section we replace the usual norm in $X$ with the equivalent one
$$\|u\|^p=\frac{1}{2}\int_{\R^{2N}\setminus(C\Omega)^2}\frac{|u(x)-u(y)|^p}{|x-y|^{N+ps}}\,dxdy + \int_\Omega |u|^p\,dx.$$
Hence, as usual, we can define the functional 
$$\mathcal{J}(u):=\frac{1}{p} \|u\|^p- \int_\Omega fu \,dx,$$
so that every critical point of $\mathcal{J}$ is a weak solution of \eqref{probco}. 

Not surprisingly, we have the following existence result:
\begin{prop}
Let $f\in L^{p'}(\Omega)$,  $s\in(0,1)$ and $p>1$. Then \eqref{probco} admits a unique solution.
\end{prop} 

\begin{proof}
First of all, the functional $\mathcal{J}$ is coercive, in fact
$$\mathcal{J}(u)\geq \frac{1}{p} \|u\|^p- C\|u\|\to \infty $$
when $\|u\|\to \infty$.
Moreover, $\mathcal{J}$ is also strictly convex, hence by the Weierstrass Theorem
it has a global minimum, which is a critical point of $\mathcal{J}$. Uniqueness follows by strict convexity.
\end{proof}

Now, we consider the problem
\begin{equation}\label{prob}
\begin{cases}
(-\Delta)^s_p u + |u|^{p-2}u=f(x,u) & $ in $ \Omega,
\\
\mathscr{N}_{s,p}u=0  & $ in $ \R^N \setminus \overline{\Omega},
\end{cases}
\end{equation}
where $f:\Omega\times \R \to \R $ a Carath\'{e}odory function such that $f(x,0)=0$ for almost every $x\in \Omega$.
In addition, we assume the following hypotheses:
\begin{itemize}
\item[$(f_1)$] there exists $a\in L^q(\Omega)$, $a\geq 0$, with $q\in ((p^*_s)',p)$, $c>0$ and $r\in (p,p^*_s)$ such that
$$|f(x,t)|\leq a(x)+c|t|^{r-1} $$
for a.e. $x \in \Omega$ and for all $t \in \R$;
\item[$(f_2)$] denoting $F(x,t)=\int_0^t f(x,\tau)d\tau $, we have 
$$\lim_{t\to \pm \infty}\frac{F(x,t)}{|t|^p}=+\infty $$
uniformly for a.e. $ x\in \Omega$;
\item[$(f_3)$] if $\sigma(x,t)=f(x,t)t-pF(x,t)$, then there exist $\vartheta\geq1$ and $\beta^* \in L^1(\Omega)$, $\beta^*\geq0$, such that
\[
\sigma(x,t_1)\leq \vartheta\sigma(x,t_2)+\beta^*(x)
\]
for a.e. $x\in \Omega$ and all $0\leq t_1 \leq t_2$ or $t_2\leq t_1 \leq 0$;
\item[$(f_4)$] 
$$\lim_{t\to 0} \frac{f(x,t)}{|t|^{p-2}t}=0 $$
uniformly for a.e. $x\in \Omega$.
\end{itemize}
As usual, in $(f_1)$ we have denoted by $p^*_s$ the fractional Sobolev exponent of order $s$, that is
\[
p^*_s=\begin{cases}
\dfrac{pN}{N-ps}& \mbox{ if }ps<N,\\
\infty &\mbox{ if }ps\geq N,
\end{cases}
\]
so that the embedding in $L^q(\Omega)$ of $W^{s,p}(\Omega)$ (and thus of $X$) is compact for every $q<p^*_s$.
 
\begin{oss}
A few comments on $(f_3)$ are mandatory. Such a condition was introduced in \cite{MP} with $\vartheta=1$. However, it is clear that assuming $\vartheta\geq1$ enlarges the set of admissible {\sl positive} (or definitely positive) functions $\sigma$'s considered in \cite{MP} (as it happens for the model case $f(x,t)=|t|^{r-2}r$). On the other hand, if $\sigma$ were negative, admitting $\vartheta<1$ would make the situation more general. However, if $(f_1)-(f_4)$ hold for some $\vartheta>0$, then $\sigma(x,t)> 0$ for a.e. $x\in\Omega$ and all $t$, at least for $|t|$ large, that is there exists $\bar t\geq0$ such that $\sigma(x,t)> 0$ for a.e. $x\in\Omega$ and all $|t|>\bar t$. Indeed, reasoning with $t$ positive, if for every $t>0$ there exists $\tau>t$ such that $\sigma(x,\tau)\leq 0$, we get $\sigma(x,t)\leq \vartheta \sigma(x,\tau)+\beta^\ast(x)\leq \beta^\ast(x)$, that is $f(x,t)t-pF(x,t)\leq \beta^\ast(x)$ for a.e. $x\in \Omega$ and all $t$. As a consequence, $(F(t)t^{-p})'\leq \beta^\ast(x)t^{-p-1}$, and so
\[
\frac{F(s)}{s^p}-\frac{F(t)}{t^p}\leq \frac{\beta^\ast(x)}{-p}\left(\frac{1}{s^p}-\frac{1}{t^p}\right)
\]
for every $t<s$. Letting $s\to +\infty$, we get a contradiction with $(f_2)$.

As a consequence, in $(f_3)$ the requirement $\vartheta\geq1$ is the most general one.
\end{oss}

Now we are ready to give the definition of a weak solution of our problem.
\begin{dfn}
Let $u \in X$. With the same assumption on $f$ as above, we say that $u$ is a weak solution of (\ref{prob}) if
\[
\begin{aligned}
\frac{1}{2} \int_{\R^{2N}\setminus(C\Omega)^2}\frac{J_p(u(x)-u(y))(v(x)-v(y))}{|x-y|^{N+ps}}\,dxdy
 &+ \int_\Omega |u|^{p-2}uv\,dx \\
& =\int_\Omega f(x,u)v\,dx
 \end{aligned}
 \]
for every $v \in X$.
\end{dfn}
With this definition, we have that any critical point of the functional ${\mathscr E}:X\to \R$ given by
$${\mathscr E}(u)=\frac{1}{p}\|u\|^p- \int_\Omega F(x,u)\,dx $$
is a weak solution of (\ref{prob}).

Our main result is the following
\begin{theorem}\label{th1}
If hypotheses $(f_1)$-$(f_4)$ hold, then problem $(\ref{prob})$ admits two non-trivial
constant sign solutions. More precisely, one solution is strictly positive in $\R^N\setminus \overline \Omega$ and the other one is strictly negative in $\R^N\setminus \overline \Omega$. In addition, if the equation in \eqref{prob} holds pointwise, each solution has strict sign in the whole of $\R^N$.
\end{theorem}
First, we introduce the functionals
$${\mathscr E}_\pm(u)=\frac{1}{p}\|u\|^p-\int_\Omega F(x,u^\pm)\,dx , $$
where $u^+$ and $u^-$ are the classical positive part and negative part of $u$. 
We want to prove that both ${\mathscr E}_\pm$ satisfies the Cerami condition, (C) for short, which states that any sequence
$(u_n)_n$ in $X$ such that $({\mathscr E}_\pm(u_n))_n$ is bounded and $(1+\|u_n\|){\mathscr E}'_\pm(u_n)\to 0$ as $n\to \infty$ 
admits a convergent subsequence.

We will also use the following inequality:
\begin{equation}\label{disug}
|x^--y^-|^p \leq |x-y|^{p-2}(x-y)(y^--x^-),
\end{equation}
for any $x,y\in \R$.

\begin{prop}\label{C}
Under the assumptions of Theorem $\ref{th1}$, ${\mathscr E}_\pm$ satisfies the $(C)$ condition.
\end{prop}
\begin{proof}
We do the proof for ${\mathscr E}_+$, the proof for ${\mathscr E}_-$ being analogous.

Let $(u_n)_n$ in $X$ be such that 
\begin{equation}\label{cer1}
|{\mathscr E}_+(u_n)|\leq M_1
\end{equation}
for some $M_1>0$ and all $n\geq 1$, and 
\begin{equation}\label{cer2}
(1+\|u_n\|){\mathscr E}'(u_n)\to 0
\end{equation}
in $X'$ as $n\to \infty$. From \eqref{cer2} we have
$$|{\mathscr E}_+'(u_n)(h)|\leq \frac{\varepsilon_nh}{1+\|u_n\|} $$
for every $h\in X$ and with $\varepsilon_n \to 0$ as $n \to \infty$, that is 
\begin{equation}\label{vs}
\begin{aligned}
\left|\int_{\R^{2N}\setminus(C\Omega)^2}\right.&\frac{J_p(u_n(x)-u_n(y))(h(x)-h(y))}{|x-y|^{N+ps}}\,dxdy+\int_\Omega |u_n|^{p-2}u_nh\,dx\\
& \left. -\int_\Omega f(x,u_n^+)h\,dx\right|\leq \frac{\varepsilon_nh}{1+\|u_n\|}. 
\end{aligned}
\end{equation}
Taking $h=u_n^-$ in \eqref{vs}, we obtain
$$\left| \int_{\R^{2N}\setminus(C\Omega)^2}\frac{J_p(u_n(x)-u_n(y))(u_n^-(x)-u_n^-(y))}{|x-y|^{N+ps}}\,dxdy
+\int_\Omega |u_n^-|^p dx\right| \leq \varepsilon_n. $$
By \eqref{disug}, we have
\begin{align*}
& \int_{\R^{2N}\setminus(C\Omega)^2}\frac{|u_n^-(x)-u_n^-(y)|^p}{|x-y|^{N+ps}}\,dxdy \\
 &\leq \int_{\R^{2N}\setminus(C\Omega)^2}\frac{J_p(u_n(x)-u_n(y))(u_n^-(x)-u_n^-(y))}{|x-y|^{N+ps}}\,dxdy, 
 \end{align*}
which leads to 
$$ \|u_n^-\|^p \leq  \varepsilon_n. $$
So, we have that
\begin{equation}\label{vs-}
u_n^- \to 0 \quad \text{in } X \quad \text{as } n\to \infty.
\end{equation}
Now, if we take $h=u_n^+$ in \eqref{vs}, we obtain
\begin{align}\label{vs+}
&-\int_{\R^{2N}\setminus(C\Omega)^2}\frac{J_p(u_n(x)-u_n(y))(u_n^+(x)-u_n^+(y))}{|x-y|^{N+ps}}\,dxdy \notag \\
&- \int_\Omega |u_n^+|^p\,dx + \int_\Omega f(x,u_n^+)u_n^+\,dx \leq \varepsilon_n.
\end{align}
From \eqref{cer1} we have
$$\int_{\R^{2N}\setminus(C\Omega)^2}\frac{|u_n(x)-u_n(y)|^p}{|x-y|^{N+ps}}\,dxdy
 +\int_\Omega |u_n|^pdx-p \int_\Omega F(x,u_n^+)\,dx \leq pM_1$$
for $M_1>0$ and all $n\geq 1$, and since
$$\int_{\R^{2N}\setminus(C\Omega)^2}\frac{J_p(u_n(x)-u_n(y))(u_n^-(x)-u_n^-(y))}{|x-y|^{N+ps}}\,dxdy
+\int_\Omega |u_n^-|^p dx\to 0$$
as $n \to \infty$, we get
\begin{align}\label{m2}
&\int_{\R^{2N}\setminus(C\Omega)^2}\frac{J_p(u_n(x)-u_n(y))(u_n^+(x)-u_n^+(y))}{|x-y|^{N+ps}}\,dxdy \notag \\
&+ \int_\Omega |u_n^+|^pdx -p \int_\Omega F(x,u_n^+)\,dx \leq M_2
\end{align}
for some $M_2>0$ and all $n\geq 1$. Adding \eqref{m2} to \eqref{vs+} we obtain
$$ \int_\Omega f(x,u_n^+)u_n^+ \,dx -p \int_\Omega F(x,u_n^+)\,dx \leq M_3 $$
for some $M_3>0$ and all $n\geq 1$, that is
\begin{equation}\label{m3}
\int_\Omega \sigma(x,u_n^+)\,dx \leq M_3.
\end{equation}

Now we want to prove that $(u_n^+)_n$ is bounded in $X$, and to do this we argue by contradiction. Passing to a subsequence
if necessary, we assume that $\|u_n^+\|\to \infty$ as $n\to \infty$. Defining $y_n=u_n^+/\|u_n^+\|$, we can assume
\begin{equation}\label{cdeb}
y_n \rightharpoonup y \quad \text{in } X \quad \text{and } y_n\to y \quad \text{in } L^q(\Omega)
\end{equation}
for every $q\in (p,p^*_s)$ and $y\geq 0$.

First we consider the case $y\neq 0$. We define $Z(y)=\lbrace x\in\Omega : y(x)=0\rbrace$, and so we have 
$\left|\Omega \setminus Z(y)\right|>0$ and $u_n^+\to \infty$ for almost every $x\in \Omega \setminus Z(y)$ as $n\to \infty$.
By hypothesis ($f_2$), we have
$$\frac{F(x,u_n^+(x))}{\|u_n^+\|^p}=\frac{F(x,u_n^+(x))}{u_n^+(x)^p}y_n(x)^p \to \infty $$
for almost every $x\in \Omega \setminus Z(y)$. By Fatou's Lemma, we have
$$\int_\Omega \liminf_{n\to \infty} \frac{F(x,u_n^+(x))}{\|u_n^+\|^p}\,dx \leq 
\liminf_{n\to \infty}\int_\Omega \frac{F(x,u_n^+(x))}{\|u_n^+\|^p}\, dx,  $$
and so 
\begin{equation}\label{fatu}
\int_\Omega \frac{F(x,u_n^+(x))}{\|u_n^+\|^p}\, dx \to \infty
\end{equation}
as $n\to \infty$.

As before, from \eqref{cer1} and \eqref{vs-} we have
$$-\frac{1}{p}\|u_n\|^p +\int_\Omega F(x,u_n^+)\,dx \leq M_4 $$
for some $M_4>0$ and $n\geq 1$. Since $\|u_n\|^p\leq 2^{p-1}(\|u_n^+\|^p+\|u_n^-\|^p)$, we obtain
$$-\frac{2^{p-1}}{p}\|u_n^+\|^p +\int_\Omega F(x,u_n^+)\,dx \leq M_5 $$
for some $M_5>0$, and so
$$-\frac{2^{p-1}}{p} +\int_\Omega \frac{F(x,u_n^+(x))}{\|u_n^+\|^p}\, dx\leq \frac{M_5}{\|u_n^+\|^p}.  $$
Passing to the limit, we have
$$\limsup_{n\to \infty}\int_\Omega \frac{F(x,u_n^+(x))}{\|u_n^+\|^p}\, dx \leq M_6  $$
for some $M_6$, which is in contradiction with \eqref{fatu}, and this concludes the case $y\neq 0$.

Now,we deal with the case $y\equiv 0$. We consider the continuous functions $\gamma_n:[0,1]\to\R$, defined as
$\gamma_n(t):={\mathscr E}_+(tu_n^+)$ with $t\in[0,1]$ and $n\geq 1$. So, we can define $t_n$ such that
\begin{equation}\label{max}
\gamma_n(t_n)=\max_{t\in[0,1]}\gamma_n(t).
\end{equation} 
Now we define $v_n:=(p\lambda)^\frac{1}{p}y_n\in X$ for $\lambda>0$. From \eqref{cdeb}, it follows that $v_n \to 0$ 
in $L^q(\Omega)$ for all $q\in (p,p^*_s)$. Starting from ($f_1$) and performing some integration, we have
$$\int_\Omega F(x,v_n(x))\,dx \leq \int_\Omega a(x)|v_n(x)|\,dx + C\int_\Omega |v_n(x)|^r\,dx, $$
and so 
\begin{equation}\label{to0}
\int_\Omega F(x,v_n(x))\,dx \to 0
\end{equation}
as $n\to \infty$. Since $\|u_n^+\|\to \infty$, there exists $n_0\geq 1$ such that 
$(p\lambda)^\frac{1}{p} /\|u_n^+\| \in(0,1)$ for all $n\geq n_0$. Then, from \eqref{max}, we have
$$\gamma_n(t_n)\geq \gamma_n\left(\frac{(p\lambda)^\frac{1}{p}}{\|u_n^+\|} \right) $$
for all $n\geq n_0$. It follows that
\begin{align*}
{\mathscr E}_+(t_nu_n^+)& \geq {\mathscr E}_+((p\lambda)^\frac{1}{p}y_n)={\mathscr E}_+(v_n) \\
& =\lambda \|y_n\|^p - \int_\Omega F(x,v_n(x))\,dx.
\end{align*} 
From \eqref{to0}, we have 
$${\mathscr E}_+(t_nu_n^+)\geq \lambda + o(1), $$ 
and since $\lambda$ is arbitrary we have
\begin{equation}\label{toinf}
{\mathscr E}_+(t_nu_n^+)\to \infty
\end{equation}  
as $n\to \infty$. Now, $0\leq t_nu_n^+\leq u_n^+$ for all $n\leq 1$, so from ($f_3$) we get
\begin{equation}\label{sigma}
\int_\Omega \sigma(x,t_nu_n^+)\,dx \leq \vartheta \int_\Omega \sigma(x,u_n^+)\,dx + \|\beta^*\|_1
\end{equation}
for all $n\geq 1$. In addition, we have ${\mathscr E}_+(0)=0$, and from \eqref{cer1}, \eqref{vs-} and \eqref{disug},
we have ${\mathscr E}_+(u_n^+)\leq M_7$ for some $M_7>0$. Together with \eqref{toinf}, this implies that $t_n\in (0,1)$ for all
$n\geq n_1\geq n_0$. Since $t_n$ is a maximum point, we also have 
\begin{align*}
0&= t_n\gamma_n'(t_n) \\
&= \int_{\R^{2N}\setminus(C\Omega)^2}\frac{J_p(t_nu_n(x)-t_nu_n(y))(t_nu_n^+(x)-t_nu_n^+(y))}{|x-y|^{N+ps}}\,dxdy  \\
&+\int_\Omega |t_nu_n^+|^p\,dx - \int_{\Omega} f(x,t_nu_n^+(x))t_nu_n^+(x)\,dx,
\end{align*}
and so, from \eqref{disug},
\begin{equation}\label{zero}
\|t_nu_n^+\|^p - \int_{\Omega} f(x,t_nu_n^+(x))t_nu_n^+(x)\,dx \leq 0.
\end{equation}
Adding \eqref{zero} to \eqref{sigma}, we get
\begin{align*}
\|t_nu_n^+\|^p - p \int_\Omega F(x,t_nu_n^+(x))\,dx \leq \vartheta \int_\Omega \sigma(x,u_n^+)\,dx + \|\beta^*\|_1,
\end{align*}
which is 
$$p{\mathscr E}_+(t_nu_n^+)\leq  \vartheta \int_\Omega \sigma(x,u_n^+)\,dx + \|\beta^*\|_1.$$
So, from \eqref{toinf}, we get
\begin{equation}\label{toinf2}
\int_\Omega \sigma(x,u_n^+)\,dx \to \infty
\end{equation}
as $n\to \infty$. Combining \eqref{m3} and \eqref{toinf2} we obtain a contradiction, and so the claim follows.

We have proved that $(u_n^+)_n$ is bounded in $X$, so from \eqref{vs-} we have that $(u_n)_n$ is bounded in $X$.
Hence, we can assume
\begin{equation}\label{cdeb2}
u_n \rightharpoonup u \quad \text{in } X \quad \text{and } u_n\to u \quad \text{in } L^q(\Omega)
\end{equation}
with $q\in (p,p^*_s)$. Taking $h=u_n-u$ in \eqref{vs}, we have
\begin{align}\label{vss}
&\|u_n\|^p -\int_{\R^{2N}\setminus(C\Omega)^2}\frac{J_p(u_n(x)-u_n(y))(u(x)-u(y))}{|x-y|^{N+ps}}\,dxdy \notag \\
&-\int_\Omega |u_n|^{p-2}u_nu \,dx- \int_{\Omega} f(x,u_n^+)(u_n-u)\,dx\leq\varepsilon_n.
\end{align}
From ($f_1$) and \eqref{cdeb2}, we have
$$\int_{\Omega} |f(x,u_n^+(x))(u_n(x)-u(x))|\,dx \to 0  $$
as $n\to \infty$. So, passing to the limit in \eqref{vss}, we get
\begin{align*}
&\|u_n\|^p -\int_{\R^{2N}\setminus(C\Omega)^2}\frac{J_p(u_n(x)-u_n(y))(u(x)-u(y))}{|x-y|^{N+ps}}\,dxdy \\
&-\int_\Omega |u_n|^{p-2}u_nu \,dx\to 0 
\end{align*}
as $n\to \infty$. This implies that $\|u_n\|^p \to \|u\|^p$, and so from the $(S)$ property it follows that $u_n \to u$ in $X$.
This concludes the proof that ${\mathscr E}_+$ satisfies the (C) condition. 
\end{proof}
We can now give the proof of Theorem \ref{th1}.

\begin{proof}[Proof of Theorem \ref{th1}]
We want to apply the Mountain Pass Theorem to ${\mathscr E}_+$. Since ${\mathscr E}_+$ satisfies the (C) condition from Proposition \ref{C},
we only have to verify the geometric conditions.

From ($f_1$) and ($f_4$), for every $\varepsilon>0$ there exists $C_\varepsilon>0$ such that
\begin{equation}\label{delta}
F(x,u)\leq \frac{\varepsilon}{p}|u|^p+ C_\varepsilon|u|^r
\end{equation}
for almost every $x\in \R^N$ and all $u \in \R$. Then, we have
\begin{align*}
{\mathscr E}_+(u)&= \frac{1}{p}\|u\|^p-\int_\Omega F(x,u^+)\,dx \\
&\geq \frac{1}{p}\|u\|^p - \frac{\varepsilon}{p}\|u\|_p^p- C_\varepsilon\|u\|_r^r \\
&\geq \frac{1-\varepsilon C_1}{p}\|u\|^p- C_2\|u\|^r.
\end{align*}
From this, if $\|u\|=\rho$ small enough, we have $\inf_{\|u\|=\rho}{\mathscr E}_+(u)>0$.

Now, we take $u\in X$ with $u>0$ and $t>0$, then
\begin{align*}
{\mathscr E}_+(u) &=\frac{t^p}{p}\|u\|^p - \int_\Omega F(x,tu)\,dx \\
&= \frac{t^p}{p}\|u\|^p - t^p\int_\Omega \frac{F(x,tu)}{(tu)^p}u^p\,dx.
\end{align*}
By Fatou's Lemma we have 
$$\int_\Omega\liminf_{t\to \infty} \frac{F(x,tu)}{(tu)^p}u^p\,dx 
\leq \liminf_{t\to \infty}\int_\Omega \frac{F(x,tu)}{(tu)^p}u^p\,dx, $$
so from ($f_2$) we have
$$ \int_\Omega \frac{F(x,tu)}{(tu)^p}u^p\,dx \to \infty$$
as $n\to \infty$. It follows that 
$${\mathscr E}_+(tu)\to -\infty $$
as $t\to \infty$, and so there exists $e \in X$ such that $\|e\| \geq \rho$ and ${\mathscr E}_+(e)>0$.

Now, we can apply the Mountain Pass Theorem to ${\mathscr E}_+$ and obtain a non-trivial critical point $u$. In particular,
we have
\begin{align*}
0&=\int_{\R^{2N}\setminus(C\Omega)^2}\frac{J_p(u_n(x)-u_n(y))(u^-(x)-u^-(y))}{|x-y|^{N+ps}}\,dxdy\\&
+\int_\Omega |u|^p\,dx -\int_\Omega f(x,u^+)u^-\,dx \\
&=\int_{\R^{2N}\setminus(C\Omega)^2}\frac{J_p(u_n(x)-u_n(y))(u^-(x)-u^-(y))}{|x-y|^{N+ps}}\,dxdy
+\int_\Omega |u|^p\,dx.
\end{align*}
From \eqref{disug}, we get 
$$ 0\geq \|u^-\|^p, $$
and so $u^-\equiv 0$. As a consequence, we have ${\mathscr E}_+(u)={\mathscr E}(u)$, and so $u\geq 0$ is a solution of
\eqref{prob}. 

Suppose that there exists $x_0\in \R^N\setminus \overline \Omega$ such that $u(x_0)=0$. Then, from Theorem \ref{boundary} we would get
\[
\int_\Omega \frac{u^{p-1}(y)}{|x-y|^{N+sp}}dy=0,
\]
so that $u=0$ in $\Omega$ and thus, using $u$ as test function in the equation, $u=0$ in $\R^N$, while $u$ is non-trivial.

Now, assume that the equation i \eqref{prob} holds pointwise and suppose by contradiction that there exists $x\in \Omega$ such that $u(x)=0$. 
From the equation we would get
$$\int_{\R^N}\frac{u(y)^{p-1}}{|x-y|^{N+ps}}\,dy=0. $$
This would imply that $u=0$ a.e. in $\R^N$, which is a contradiction since the solution is non-trivial. It follows that $u>0$ in $\R^N$.  

Arguing in the same way for ${\mathscr E}_-$, we can find a non-trivial negative solution for \eqref{prob}.
\end{proof}
\medskip

\textit{Some open questions}.
\begin{enumerate}
\item {\sl Is any solution of problem \eqref{prob} continuous in $\R^N$?} In the Dirichlet case ``$u=0$ on $\R^N\setminus \Omega$'', this last condition helps significantly in obtaining the desired regularity. In our case,  we believe this result is true, but at the moment we are not able to prove it.
\item {\sl Is it true that any solution of problem \eqref{prob} solves the equation $(-\Delta)_p^su=f(x,u)$ a.e. in $\Omega$?} Of course, if $f$ is continuous and the solution is regular, this would be true.
\end{enumerate}

\section*{Acknowledgments}
The first author is Member of the Gruppo Nazionale per l'Analisi Matematica, la Probabilit\`a  a e le loro Applicazioni (GNAMPA) of the Istituto Nazionale di Alta Matematica (INdAM) ``F. Severi''. He is supported by the MIUR National Research Project {\sl Variational methods, with applications to problems in mathematical physics and geometry} (2015KB9WPT\underline{\ }009) and by the FFABR ``Fondo per il finanziamento delle attivit\`a base di ricerca'' 2017.

The second author is is Member of the Gruppo Nazionale per l'Analisi Matematica, la Probabilit\`a  a e le loro Applicazioni (GNAMPA) of the Istituto Nazionale di Alta Matematica (INdAM) ``F. Severi''.

\end{document}